\documentclass{amsart}
\usepackage{graphicx}
\usepackage{latexsym}
\usepackage{color}
\usepackage{amsfonts}
\usepackage[all]{xy}
\usepackage{amssymb, mathrsfs, amsfonts, amsmath}
\usepackage{amsbsy}
\usepackage{amsfonts}
\newtheorem{corollary}{Corollary}

\newtheorem{lemma}{Lemma}
\newtheorem{proposition}{Proposition}
\newtheorem{remark}{Remark}
\newtheorem{theorem}{Theorem}
\newtheorem{example}{Example}
\numberwithin{equation}{section}

\newcommand{\be}{\begin{equation}}
	\newcommand{\ee}{\end{equation}}
\newcommand{\ben}{\begin{enumerate}}
	\newcommand{\een}{\end{enumerate}}
\newcommand{\beq}{\begin{eqnarray}}
	\newcommand{\eeq}{\end{eqnarray}}
\newcommand{\beqn}{\begin{eqnarray*}}
	\newcommand{\eeqn}{\end{eqnarray*}}

\begin{document}
	\title{On Douglas warped product metrics}
	\author{Newton Mayer  Sol\'orzano Chávez}
	
	\address{Universidade Federal da Integra\c{c}\~{a}o Latino-Americana  -  Avenida Silvio Am\'erico Sasdelli, 1842 - Vila A, Edif\'icio Comercial Lorivo - CEP: 85866-000 - Caixa Postal 2044 - Foz do Igua\c{c}u - Paran\'a.}
	\email{nmayer159@gmail.com}

	\begin{abstract}
We study the new warped metric proposed by P. Marcal and Z. Shen. We obtain the differential equation of such metrics with vanishing Douglas curvature. By solving this equation, we obtain all Douglas warped product metrics. We show that Landsberg and Berwald warped product metrics are equivalent. We classify Douglas Ricci-flat metrics. Examples are included.
	\end{abstract}
	
	\keywords{Finsler metric, warped product, Douglas metric.}   
	\subjclass[2020]{53B40, 53C60}
	\date{\today}

	\maketitle
\section{Introduction}

A Finsler metric on a manifold $M$ is a {\em Douglas metric} if its Douglas curvature vanishes identically. 
The Douglas curvature was introduced by J.
Douglas \cite{D} in 1927. Its importance in Finsler geometry is due to the fact that  it is a projective invariant. Namely, if two Finsler metrics $F$ and
$\bar{F}$ are projectively equivalent, then $F$ and $\bar{F}$ have the same Douglas curvature. 

The class of Douglas metrics contains all Riemannian metrics and the locally projectively flat Finsler metrics.  However, there are
many Douglas metrics which are not Riemannian. There are also 
many Douglas metrics which are not locally projectively flat.

The warped product metric was introduced by Bishop and O'Neil \cite{bishop1969} to study Riemannian manifolds of negative curvature, as a generalization of Riemannian product metrics. The notion of warped products was  extended to the case of Finsler manifolds  \cite{Zhao2018,Kozma2001}. These metrics are called {\em Finsler warped product metrics}. In \cite{Zhao2018}, it was observed that spherically symmetric Finsler metrics are actually warped product metrics. In \cite{Liu2019}, the authors used this observation and gave all the complete characterization of Douglas Finsler warped product metrics and some new Douglas metrics of this type were produced by using known spherically symmetric Douglas metrics given in \cite{MoZoTe2013}.

In \cite{marcal2020}, the authors considered a new class of Finsler metrics using the warped product notion introduced by Chen, S. and Zhao \cite{Zhao2018}, with another ``warping", one that is consistent with static spacetime.
 They gave the PDE characterization for the proposed metrics to be Ricci-flat and they explicitly constructed two non-Riemannian examples. In this paper, we characterize such metrics with vanishing Douglas curvature in terms of a differential equation.  Then, by solving this equation, we obtain all Douglas warped product  metrics (Theorem 2). We also obtain the Berwald curvature and the Landsberg curvature, concluding  Landsberg and Berwald warped product metrics are equivalent (Corollary \ref{corollaryD}).
Additionally, we characterize the Ricci-flat Douglas warped product metrics (Corollary \ref{DouglasRicci}).

\

\section{Preliminaries}

In this section, we  give some notations, definitions and lemmas that will be used in the proof of our main results.
Let $M$ be a manifold and let $TM=\cup_{x\in M}T_xM$ be the tangent
bundle of $M$, where $T_xM$ is the tangent space at $x\in M$. We
set $TM_o:=TM\setminus\{0\}$ where $\{0\}$ stands for
$\left\{(x,\,0)|\, x\in M,\, 0\in T_xM\right\}$. A {\em Finsler
metric} on $M$ is a function $F:TM\to [0,\,\infty)$ with the
following properties

(a) $F$ is $C^{\infty}$ on $TM_o$;

(b) At each point $x\in M$, the restriction $F_x:=F|_{T_xM}$ is a
Minkowski norm on $T_xM$.

\

Let  $\mathbb{B}^n(\rho)\subset\mathbb{R}^n$ the $n$ dimensional ($n\geq 2$) open ball of radius $ \rho $ and centered at the origin. Set 
 $M=\mathbb{R}\times \mathbb{B}^n(\rho),$ with coordinates on $ TM $
\begin{align}
	x&=(x^0, \overline{x}), \quad \overline{x}=(x^1,\ldots,x^n),\label{coordx}\\
	y&=(y^0, \overline{y}), \quad \overline{y}=(y^1,\ldots,y^n).\label{coordy}
\end{align} 
We introduce the notation 
\begin{align}\label{rs} 
	z&:= \frac{y^0}{\vert\overline{y}\vert}, &r&:=|\overline{x}|,  &s&:=\frac{\langle \overline{x},\,\overline{y}\rangle}{\vert\overline{y}\vert},
\end{align} 
where $|\,.\,|$ and $\langle \,,\rangle$ are the 
standard Euclidean norm and inner product on $\mathbb{R}^n$. 

Throughout our work, the following convention for indices is adopted: 
\begin{align}
	0\leq &A, B, \ldots \leq n;\\
	1\leq &i,j,\ldots \leq n.
\end{align}  

Consider the Finsler metric $ F $ defined on $ M $ such that \begin{align}\label{eq:sph2}
	F((x^0,{O}\overline{x}),(y^0,{O}\overline{y}))=F((x^0,\overline{x}),(y^0,\overline{y}))\end{align} 
for every orthogonal $ n\times n $ matrix $ {O}. $

 Inspired by \cite{Z}, (see also \cite{HM2}) we show the following:

\begin{lemma}
	A Finsler metric $ F, $ defined on $ M=\mathbb{R}\times\mathbb{B}^n(\rho), $ satisfies \eqref{eq:sph2} if, and only if, there exists a differentiable function $ \phi:\mathbb{R}^4\rightarrow \mathbb{R} $ such that
	\[F(x,y)=\vert \overline{y}\vert\sqrt{\phi\left(x^0,\frac{y^0}{\vert \overline{y}\vert},\vert \overline{x}\vert,\frac{\langle\overline{x},\overline{y}\rangle}{\vert\overline{y}\vert}\right)}.\]

\end{lemma}
\begin{proof}
	Suppose that $ 	F((x^0,O\overline{x}),(y^0,O\overline{y}))=F((x^0,\overline{x}),(y^0,\overline{y})). $ Denote by $ e_A $ the $ n+1-$dimensional vector with $ 1 $ in the $ A^{th} $ entry and zeros elsewhere.

	For $ \overline{x}\neq 0, $ put
	\begin{align}
		(0,\nu_1)=&\left(0,\frac{\overline{x}}{\vert\overline{x}\vert}\right), & (0,\nu_2)=&\left(0,\frac{\overline{y} -\frac{\langle \overline{x},\overline{y}\rangle}{\vert \overline{x}\vert^2}\overline{x} }{\left\vert\overline{y} - \frac{\langle \overline{x},\overline{y}\rangle}{\vert \overline{x}\vert^2}\overline{x} \right\vert}\right).
	\end{align}
	Then $ \nu_1 $ and $ \nu_2 $ are orthonormal vectors in $ \mathbb{R}^n. $ It follows that there exists an ${O}= {O}({\overline{x},\overline{y}})\in \mathcal{O}(n) $ such that 
	\begin{align}\label{Aepsilon}
		(0,O\nu_1)=&e_1,&  (0,O\nu_2)=&  e_2.
	\end{align}
	Using \eqref{Aepsilon}, we have that
	
	\[(x^0,O\overline{x})=(x^0,\vert\overline{x}\vert O\nu_ 1)=x^0e_0 + \vert\overline{x}\vert e_1, \]
	and
	
	\begin{align*}
		(y^0,O\overline{y})=&\left(y^0,O\left({\left\vert\overline{y} -\frac{\langle \overline{x},\overline{y}\rangle}{\vert \overline{x}\vert^2}\overline{x} \right\vert}\nu_2+ \frac{\langle \overline{x},\overline{y}\rangle}{\vert \overline{x}\vert^2}\overline{x} \right)\right)\\
		=&\left(y^0,\frac{\langle \overline{x},\overline{y}\rangle}{\vert \overline{x}\vert^2}O\overline{x} + \frac{\sqrt{\vert\overline{x}\vert^2\vert\overline{y}\vert^2-\langle\overline{x},\overline{y}\rangle^2}}{\vert\overline{x}\vert}O\nu_2\right)\\
		=&y^0e_0+\frac{\langle \overline{x},\overline{y}\rangle}{\vert \overline{x}\vert}e_1 + \frac{\sqrt{\vert\overline{x}\vert^2\vert\overline{y}\vert^2-\langle\overline{x},\overline{y}\rangle^2}}{\vert\overline{x}\vert}e_2,
	\end{align*}
	for any $ x^0, y^0. $	
	Applying the condition \eqref{eq:sph2} of $ F $ we obtain
	\begin{align*}
		F(x,y)=&F\left(x^0,R\overline{x},y^0,R\overline{y}\right)\\
		=&F\left(x^0,\vert\overline{x}\vert,\overline{0};y^0,\frac{\langle \overline{x},\overline{y}\rangle}{\vert \overline{x}\vert},\frac{\sqrt{\vert\overline{x}\vert^2\vert\overline{y}\vert^2-\langle\overline{x},\overline{y}\rangle^2}}{\vert\overline{x}\vert},\overline{0}\right).
	\end{align*}
	By the homogeneity of $ F $ with respect to $ y, $ we have,
	
	\begin{align*}
		\lambda F(x,y)=&F(x,\lambda y)\\
		=&
		F\left(x^0,\vert\overline{x}\vert,\overline{0};\lambda y^0,\frac{\langle \overline{x},\lambda \overline{y}\rangle}{\vert \overline{x}\vert},\frac{\sqrt{\vert\overline{x}\vert^2\vert\lambda \overline{y}\vert^2-\langle\overline{x},\lambda\overline{y}\rangle^2}}{\vert\overline{x}\vert},\overline{0}\right),
	\end{align*}

	in particular, for $ \lambda=	\frac{1}{\vert\overline{y}\vert} $, we have that $ F(x,y) $  can be expressed as a function of $ x^0,\frac{ y^0}{\vert\overline{y}\vert},\vert\overline{x}\vert$ and $\frac{\langle\overline{x},\overline{y}\rangle}{\vert\overline{y}\vert}. $
	
	Then we define $ \phi $ as follow:
	
	\begin{align*}
		\vert\overline{y}\vert^2\phi\left(x^0,\frac{y^0}{\vert \overline{y}\vert},\vert \overline{x}\vert,\frac{\langle\overline{x},\overline{y}\rangle}{\vert\overline{y}\vert}\right): =& F^2\left(x^0,\vert\overline{x}\vert,\overline{0}; \frac{y^0}{	\vert\overline{y}\vert},\frac{1}{\vert\overline{x}\vert}\frac{\langle \overline{x}, \overline{y}\rangle}{	\vert\overline{y}\vert},\frac{1}{\vert\overline{x}\vert}
		\sqrt{\vert\overline{x}\vert^2-\frac{\langle\overline{x},\overline{y}\rangle^2}{	\vert\overline{y}\vert^2}},\overline{0}\right)\\
		=&F^2(x,y).
	\end{align*}
	
	Conversely, assume that 
	\[F(x,y)=\vert \overline{y}\vert\sqrt{\phi\left(x^0,\frac{y^0}{\vert \overline{y}\vert},\vert \overline{x}\vert,\frac{\langle\overline{x},\overline{y}\rangle}{\vert\overline{y}\vert}\right)},\]
	for some function $ \phi:\mathbb{R}^{4} \rightarrow \mathbb{R}$. Then, clearly \begin{align*}
		F(x^0,R\overline{x},y^0,R\overline{y})=&\vert R\overline{y}\vert\sqrt{\phi\left(x^0,\frac{y^0}{\vert R\overline{y}\vert},\vert R\overline{x}\vert,\frac{\langle R\overline{x},R\overline{y}\rangle}{\vert R\overline{y}\vert}\right)}\\
		=&\vert \overline{y}\vert\sqrt{\phi\left(x^0,\frac{y^0}{\vert \overline{y}\vert},\vert \overline{x}\vert,\frac{\langle \overline{x},\overline{y}\rangle}{\vert \overline{y}\vert}\right)}=F(x,y).
	\end{align*}
\end{proof}
This class of metrics generalizes the metrics considered in  \cite{Zhao2018,Kozma2001} and \cite{marcal2020}.

\

\

The matrix $ g_{AB}=\frac{1}{2} [F^2]_{y^iy^j},$ is given by
\[ 
\left(g_{AB}\right)=
\left(
\begin{array}{c|c}
	\frac{1}{2}\phi_{zz} & \frac{1}{2}\Omega_z\frac{y^j}{\vert \overline{y}\vert} + \frac{1}{2}\phi_{sz}x^j  \\
	\hline
	\frac{1}{2}\Omega_z\frac{y^i}{\vert \overline{y}\vert} + \frac{1}{2}\phi_{sz}x^i & \frac{1}{2}\Omega\delta_{ij}-\frac{1}{2}(z\Omega_z+s\Omega_s)\frac{y^i}{\vert \overline{y}\vert}\frac{y^j}{\vert \overline{y}\vert}+\frac{1}{2}\Omega_s(x^i\frac{y^j}{\vert \overline{y}\vert} + x^j\frac{y^i}{\vert \overline{y}\vert}) + \frac{1}{2}\phi_{ss}x^ix^j
\end{array}
\right)
\]
where,
\begin{align*}
	\Omega&=2\phi-z\phi_z-s\phi_s,\\
	\Omega_z&=\phi_z-z\phi_{zz}-s\phi_{sz},\\
	\Omega_s&=\phi_s-z\phi_{sz}-s\phi_{ss}.
\end{align*}

Through this work we consider the particular case when $ \phi $ depends only on $ z $ and $ r, $ where $ z $ and $ r $ are given by \eqref{rs}.

\noindent 

A Finsler metric on a manifold $N$ is called a {\em Douglas metric} if its geodesic coefficients $G^i=G^i(x,\,y)$ are given in the following form
$$
G^i=\frac 12\Gamma^i_{jk}(x)y^jy^k+P(x,\,y)y^i,
$$
where $\Gamma^i_{jk}(x)$ are functions on $N$, in local coordinates,  and $P(x,\,y)$
is a local positively $y$-homogeneous function of degree one. Douglas
metrics are also characterized by vanishing Douglas curvature, i.e., $D=0$.

Considering the Finsler metric introduced in \cite{marcal2020}:

\begin{align}\label{defdeF}
	F(x,y)=\vert \overline{y}\vert\sqrt{\phi(z,r)}
\end{align}
where $ z=\frac{y^0}{\vert \overline{y}\vert}, $ $ r=\vert \overline{x}\vert. $

The matrix, $ g_{AB}=\frac{1}{2} [F^2]_{y^iy^j},$ is given by
\[
\left(g_{AB}\right)=
\left(
\begin{array}{c|c}
	\frac{1}{2}\phi_{zz} & \frac{1}{2}\Omega_z\frac{y^i}{\vert \overline{y}\vert}  \\
	\hline
	\frac{1}{2}\Omega_z\frac{y^i}{\vert \overline{y}\vert}  & \frac{1}{2}\Omega\delta_{ij}-\frac{1}{2}z\Omega_z\frac{y^i}{\vert \overline{y}\vert}\frac{y^j}{\vert \overline{y}\vert}
\end{array}
\right),
\]
where,
\begin{align}\label{defOmega}
	\Omega&=2\phi-z\phi_z.
\end{align}
Then,
\begin{align*}
	\det(g_{AB})=\frac{1}{2^{n+1}}\Omega^{n-1}\Lambda,
\end{align*}
where
\begin{align}\label{defLambda}
	\Lambda&=2\phi\phi_{zz}-\phi_z^2.
\end{align}

We recall the following,
\begin{proposition}[\cite{marcal2020}]\label{propdconvex}
	$ F=\vert\overline{y}\vert\sqrt{\phi(z,r)} $ is strongly convex if, and only if, $ \Omega, \Lambda > 0. $
\end{proposition}

\

\

The inverse $ g^{AB} $ is given by

\[
\left(g^{AB}\right)=
\left(
\begin{array}{c|c}
	\frac{2}{\Lambda}\left(\Omega-z\Omega_z\right) & -\frac{2}{\Lambda}\Omega_z\frac{y^i}{\vert \overline{y}\vert}  \\
	\hline
	-\frac{2}{\Lambda}\Omega_z\frac{y^i}{\vert \overline{y}\vert}  & \frac{2}{\Omega}\delta^{ij} +\frac{2\phi_z(\phi_z-z\phi_{zz})}{\Omega\Lambda}\frac{y^i}{\vert\overline{y}\vert}\frac{y^j}{\vert\overline{y}\vert}
\end{array}
\right),
\]
and the spray coefficients $ G^C=\frac{1}{4}g^{CA}\left([F^2]_{y^Ax^B}y^B-[F^2]_{x^A}\right) $ are

\begin{align}
	G^0=&\vert \overline{y}\vert \langle\overline{x},\overline{y}\rangle(U+zV),\label{eq:G0}\\
	G^i=&\langle\overline{x},\overline{y}\rangle(V+W)y^i-\vert \overline{y}\vert^2Wx^i,\label{eq:Gi}
\end{align}
where
\begin{align*}
	U:=&\frac{1}{2r\Lambda}(2\phi\phi_{zr}-\phi_z\phi_r),\\
	V:=&\frac{1}{2r\Lambda}(\phi_r\phi_{zz}-\phi_z\phi_{zr}),\\
	W:=&\frac{1}{2r\Omega}\phi_r.
\end{align*}

\section{Douglas curvature}

In \cite{D}, Douglas introduced the local functions 
 $D_j{}^i{}_{kl}$ on
${T}N^n$ defined by
\begin{equation}
D_j{}^i{}_{kl}:=\frac{\partial^3}{\partial y^j\partial y^k\partial
y^l}\left(G^i-\frac 1{n+1}\sum_m \frac{\partial G^m}{\partial
y^m}y^i\right),  \label{3.1}
\end{equation}
in local coordinates $x^1,...,x^n$ and $y=\sum_i y^i \partial/\partial x^i$. 
These functions are called  {\em Douglas curvature} \cite{D} and a Finsler metric 
 $F$ is said to be a {\em Douglas metric} if $D_j{}^i{}_{kl}=0$.
In  our next result, we obtain the Douglas curvature of \eqref{defdeF}.

\

\noindent

\begin{theorem}\label{Dcurv} Let $F=\vert\overline{y}\vert\sqrt{\phi(z,r)}$ be a  Finsler metric defined on $ M $, where $ z=\frac{y^0}{\vert\overline{y}\vert}, $ $r=\vert\overline{x}\vert$, and  $TM $  with coordinates   \eqref{coordx}, \eqref{coordy}.  Then the
Douglas curvature of $F$ is given by
\begin{align}
	D^0_{000}=&\frac{s}{u} R_{zzz},\nonumber\\
	D^0_{00l}=&\frac{1}{u}\left(R_{zz}\right)x^i-\frac{s}{u}\left(R_{zz}+zR_{zzz}\right)y^i,\nonumber\\
	D^0_{0kl}=&sz\left({z}R_{zzz}+{3}{R_{zz}}\right)\frac{y^ky^l}{u^3} -zR_{zz}\left(\frac{x^ky^l+x^ly^k}{u^2}\right)-szR_{zz}\frac{\delta^{kl}}{u},\nonumber\\
	D^0_{jkl}=&\frac{s}{u^4}\left[R-zR_z-2z^2R_{zz}-\frac{1}{3}z^3R_{zzz}\right]{(y^jy^ky^l)_{\overrightarrow{jkl}}}\nonumber\\
	&+\frac{1}{u^3}\left[-R+zR_z+z^2R_{zz}\right](y^jy^kx^l)_{\overrightarrow{jkl}}+\frac{1}{u}\left[R-zR_z\right](\delta^{jk}x^l)_{\overrightarrow{jkl}}\nonumber\\
	&+\frac{s}{u^2} \left[-R+zR_z+z^2R_{zz}\right](\delta^{jk}y^l)_{\overrightarrow{jkl}}\nonumber\\
	D^i_{000}=& \frac{s}{u^2}T_{zzz}y^i-\frac{1}{u}W_{zzz}x^i\nonumber\\
	D^i_{00l}=&-\frac{s}{u^3}(2T_{zz}+zT_{zzz})y^iy^l + \frac{1}{u^2}T_{zz}y^ix^l + \frac{z}{u^2}W_{zzz}y^lx^i + \frac{s}{u}T_{zz}\delta^{il},\nonumber\\
	D^i_{0kl}=&\frac{s}{u^4}\left(3T_z+5zT_{zz} + z^2T_{zzz}\right)y^iy^ky^l-\frac{s}{u^2}(T_z+zT_{zz})(\delta^{kl}y^i)_{\overrightarrow{ikl}}\nonumber\\
	&-\frac{1}{u^3}(T_z+zT_{zz})(x^ly^ky^i+x^ky^ly^i) + \frac{1}{u}T_z(x^l\delta^{ik} + x^k\delta^{jl})\nonumber\\
	&+\frac{1}{u^3}(W_z-zW_{zz}-z^2W_{zzz})x^iy^ky^l-\frac{1}{u}(W_z-zW_{zz})x^i\delta^{kl},\nonumber\\
	D^i_{jkl}=&\frac{s}{u^5}(-15zT_z-9z^2T_{zz}-z^3T_{zzz})y^iy^jy^ky^l \label{Dijkl}\\
	&+(3zT_z+z^2T_{zz})\left(\frac{1}{u^4}y^ix^jy^ky^l +\frac{s}{u^3}\delta^{ij}y^ky^l+\frac{s}{u^3}\delta^{jl}y^ky^i \right)_{\overrightarrow{jkl}} \nonumber\\
	& - \frac{1}{u^2}zT_z\left((x^jy^i+us\delta^{ij})\delta^{kl}+(x^j\delta^{ik}+x^k\delta^{ij})y^l\right)_{\overrightarrow{jkl}}\nonumber\\
	&- \frac{1}{u^4}(3zW_z-3z^2W_{zz}-z^3W_{zzz})x^iy^jy^ky^l+ \frac{z}{u^2}(W_z-zW_{zz})x^i(y^j\delta^{kl})_{\overrightarrow{jkl}} ,\nonumber
\end{align}
where $ s=\frac{\langle\overline{x},\overline{y}\rangle}{\vert\overline{y}\vert}, $ $u=\vert\overline{y}\vert $,
\begin{align}\label{defRT}
	R&=U-\frac{z}{n+2}(U_z+(n-1)W),& T&=\frac{1}{n+2}{(3W-U_z)}, 
\end{align}

\begin{align}\label{defUVW}
	U=&\frac{2\phi\phi_{zr}-\phi_z\phi_r}{2r(2\phi\phi_{zz}-\phi_z^2)},&
	V=&\frac{\phi_r\phi_{zz}-\phi_z\phi_{zr}}{2r(2\phi\phi_{zz}-\phi_z^2)},&
	W=&\frac{\phi_r}{2r(2\phi-z\phi_z)},
\end{align}
 and $(.)_{\overrightarrow{jkl}}$ denotes cyclic permutation.
\end{theorem}

\begin{proof}
	Note that $ \vert\overline{y}\vert s(V+W) $ is positive homogeneous of degree 1 on $ (y^0,\overline{y}). $ From Euler's theorem for homogeneous functions, 
	\[	y^0\frac{\partial}{\partial y^0}\left( \vert\overline{y}\vert s(V+W)\right)+y^i\frac{\partial}{\partial y^i}\left( \vert\overline{y}\vert s(V+W)\right) = \vert\overline{y}\vert s(V+W),\]
	then
	\begin{align}\label{partialGA}
		\frac{\partial G^A}{\partial y^A}=\vert\overline{y}\vert s\left[ U_z+(n+1)(V+W)+V-2W\right].
	\end{align}
Using \eqref{partialGA} we obtain

\begin{align}
	G^0-\frac{y^0}{n+2}\frac{\partial G^A}{\partial y^A}=&u^2s\left[U-\frac{z}{n+2}(U_z+(n-1)W)\right];\label{eq:Giminus}\\
	=&u^2sR\nonumber\\
	G^i-\frac{y^i}{n+2}\frac{\partial G^A}{\partial y^A}=&us\left[\frac{3W-U_z}{n+2}\right]y^i-u^2Wx^i\label{eq:Giminusi}\\
	=&usTy^i-u^2Wx^i.\nonumber
\end{align}

Substituting \eqref{eq:Giminus}-\eqref{eq:Giminusi} into \eqref{3.1},

\begin{align}
	D_B{}^0{}_{CD}=&\frac{\partial^3}{\partial y^B\partial y^C\partial y^D}(u^2sR), \label{use1}\\
	=&(u^2s)_{y^By^Cy^D}R+\left((u^2s)_{y^By^C}R_{y^D} + (u^2s)_{y^B}R_{y^Cy^D}\right)_{\overrightarrow{BCD}}  +u^2sR_{\small{y^By^Cy^D}}; \nonumber\\
	D_B{}^i{}_{CD}=&\frac{\partial^3}{\partial y^B\partial y^C\partial y^D}(usT-u^2Wx^i)\label{use2}\\
	=&(us)_{y^By^Cy^D}T+\left((us)_{y^By^C}T_{y^D} +(us)_{y^B}T_{y^Cy^D}\right)_{\overrightarrow{ BCD}}  +usT_{y^By^Cy^D}\nonumber\\
	&-\left((u^2)_{y^By^C}W_{y^D}- (u^2)_{y^B}W_{y^Cy^D}\right)_{\overrightarrow{BCD}}  -u^2W_{y^By^Cy^D}).\nonumber
\end{align}
For any function $ \Theta=\Theta(z) $ we have,
\begin{align}
	\frac{\partial }{\partial y^0}\Theta = & \frac{1}{u}\Theta_z, \label{used01}\\
	\frac{\partial }{\partial y^l}\Theta = & -\frac{z}{u^2}\Theta_zy^l,\\
	\frac{\partial^2 }{\partial y^k\partial y^l}\Theta = &(z^2\Theta_{zz}+3z\Theta_z)\frac{y^jy^k}{u^4} - z\Theta_z\frac{\delta^{jk}}{u^2},\\
	\frac{\partial^3 }{\partial y^j\partial y^k\partial y^l}\Theta = &-\left(15z\Theta_z+9z^2\Theta_{zz}+z^3\Theta_{zzz}\right)\frac{y^jy^ky^l}{u^6} + (z^2\Theta_{zz}+3z\Theta_z)\frac{\left(\delta^{jk}y^l\right)_{\overrightarrow{jkl}}}{u^4}\label{used02}.
\end{align}
Thus, using \eqref{used01} - \eqref{used02} into \eqref{use1} and \eqref{use2}, we conclude the proof.

\end{proof}

\begin{lemma}
 Let $F=\vert\overline{y}\vert\sqrt{\phi(z,r)}$ be a  Finsler metric defined on $ M $, where $ z=\frac{y^0}{\vert\overline{y}\vert}, $ $r=\vert\overline{x}\vert$, and  $TM $  with coordinates   \eqref{coordx}, \eqref{coordy}. Then $F$ has vanishing Douglas curvature if, and only if,  $ \phi $ satisfies
 \begin{align}
 	R-zR_z&=0,\label{eqR}\\
 	T_z&=0,\label{eqT}\\
 	W_z-zW_{zz}&=0,\label{eqW}
 \end{align}
where $ R, T $ and $ W $ are in given in \eqref{defRT}-\eqref{defUVW}.
\end{lemma}
\begin{proof}
		Evidently \eqref{eqR}, \eqref{eqT} and \eqref{eqW} imply that $ D^A_{BCD}=0 $. On the other hand, the functions $ R, T, W $ and their derivatives in $ z $, are functions of $ r=\vert\overline{x}\vert $ and $ z=\frac{y^0}{\vert\overline{y}\vert} $. Then we can choose the coordinates of $\overline{x}$ and $\overline{y}  $ such that their norms remain unchanged.
	In \eqref{Dijkl} we consider $ \overline{x}, $  $ \overline{y} $ such that $ s=0 $ and $ i,j,k,l  $ all distinct, then 
	\begin{align*}
		D^i_{jkl}=\frac{1}{u^4}(3zT_z+z^2T_{zz})y^i(x^jy^ky^l)_{\overrightarrow{jkl}} - \frac{1}{u^4}(3zW_z-3z^2W_{zz}-z^3W_{zzz})x^iy^jy^ky^l = 0,
	\end{align*}
	for all $ y^i $. Hence we conclude that 
	\begin{align}
		3zW_z-3z^2W_{zz}-z^3W_{zzz}&=0,\nonumber\\
		3zT_z+z^2T_{zz}&=0.\nonumber
	\end{align} 
	Now we consider $ i,j,k,l $ all distinct, and $ s\neq 0, $ in \eqref{Dijkl}, then $ D^i_{jkl}=0 $ implies \[15T_z+9zT_{zz}+z^2T_{zzz}=0.\]
	
	Finally, we consider $ j=k=l $ and $ i\neq j, $ in \eqref{Dijkl}, then $ D^i_{jjj}=0 $ is equivalent to $T_zy^i(x^j) + (zW_{zz}-W_z)x^iy^j=0,$
	for every $ y^i, y^j \in \mathbb{R}. $ Hence $ W_z-zW_{zz}=0 $ and $ T_z=0. $ Similar arguments on $ D^0_{jkl}=D^0_{00l}=0, $ imply that $ R-zR_z=0. $ 
\end{proof}

\begin{theorem}\label{theoDouglas}
Let $F=\vert\overline{y}\vert\sqrt{\phi(z,r)}$ be a   Finsler metric defined on $ M $, where $ z=\frac{y^0}{\vert\overline{y}\vert}, $ $r=\vert\overline{x}\vert$, and  $TM $  with coordinates   \eqref{coordx}, \eqref{coordy}. Then, $F$ has vanishing Douglas curvature if, and only if, either $ F $ is a Randers metric of the form
	\begin{align}\label{Fgneq0}
 		F(x,y)=f(r)\sqrt{g(r)(y^0)^2 + \vert\overline{y}\vert^2} + b y^0,
 	\end{align} where $ f,g:J\subset \mathbb{R}\rightarrow \mathbb{R} $ are positive differentiable functions, and $ b $ is constant such that $ \vert b\vert^2<{f^2(r)}g(r)$, or $ F $ is of the form
\begin{align}\label{Fgeq0}
 		F(x,y)=\vert\overline{y}\vert {(h(r))^{-1}} {G}\left(h(r)z\right),
\end{align}
 where $ h $ and $ G $ are arbitrary positive functions such that,
\begin{align}
	{G}-t({G})'>&0\label{condGG},\\
	({G})''>&0.\label{condGG2}
\end{align}

\end{theorem}

\begin{proof}
For simplicity, we define $ \Phi=\sqrt{\phi} $, then \eqref{defOmega}, \eqref{defLambda} and \eqref{defUVW} become
\begin{align}\label{defUW}
	U&=\frac{\Phi_{rz}}{2r\Phi_{zz}},& W&=\frac{\Phi_r}{2r(\Phi-z\Phi_z)}, & \Omega&=2\Phi(\Phi-z\Phi_z), & \Lambda&=4\Phi^3\Phi_{zz},
\end{align}

 From \eqref{eqW}, there exist real differentiable functions $ h_1, h_2:J\subset \mathbb{R}\rightarrow \mathbb{R} $ such that 
 \begin{align}\label{Wh1h2} W=\frac{1}{2r}\left(h_1(r)\frac{z^2}{2}+h_2(r)\right). 
 \end{align} 
 From \eqref{defRT}, \eqref{eqR} and \eqref{eqT} we have $ U-zU_z+z^2W_z=0, $ which implies that there exists $ h_3:J\subset \mathbb{R}\rightarrow \mathbb{R} $ such that \begin{align}\label{Uh1h3} 
 	U=\frac{z}{2r}\left(h_1(r)\frac{z^2}{2} + h_3(r)\right). 
 \end{align}
 Then, from \eqref{defUW}, \eqref{Wh1h2} and \eqref{Uh1h3}, we have that $ \Phi(r,z) $ satisfies the system
 \begin{align}
 	{\Phi_r}&=\left(h_1(r)\frac{z^2}{2}+h_2(r)\right)(\Phi-z\Phi_z),\label{eqdeW}\\
 	\Phi_{rz}&=z\left(h_1(r)\frac{z^2}{2}+h_3(r)\right)\Phi_{zz}\label{eqdeU}.
 \end{align} 
Taking the derivative of \eqref{eqdeW} with respect to $ z $, 
\begin{align}\label{eqPhirz}
	\Phi_{rz}&=-z\left(h_1(r)\frac{z^2}{2} + h_2(r)\right)\Phi_{zz} + zh_1(r)(\Phi-z\Phi_z).
\end{align}
Observe that, if $ h_1(r)=0, $ then, from \eqref{eqdeU} and \eqref{eqPhirz},  $ h_2(r)=-h_3(r), $ and equation \eqref{eqdeW} can be rewritten as
\begin{align*}
	\left[e^{-\int h_2(r)dr}\Phi\right]_r + zh_2(r)\left[e^{-\int h_2(r)dr}\Phi\right]_z=0.
\end{align*}
Then, there exists a positive differentiable function  $ {G}:\mathbb{R}\rightarrow\mathbb{R} $ such that  
\begin{align}\label{phiwithG}
		\phi(r,z)=\Phi^2(r,z)=&(h(r))^{-2}{G^2}\left(h(r)z\right),
\end{align} where $ h(r)=e^{-\int h_2(r)\operatorname{dr}}. $

If $ h_1(r)\neq 0 $ and $ h_2(r)+h_3(r)\neq 0, $ from \eqref{eqdeU} and \eqref{eqPhirz}, we have,
\begin{align}\label{eqPhi}
	(g(r)z^2+1)\Phi_{zz}&=g(r)(\Phi-z\Phi_{z}),
\end{align}
where $ g=\frac{h_1}{h_2+h_3}. $ From Proposition \ref{propdconvex}, we conclude that $ g(r)>0.$
Solving   \eqref{eqPhi}, we have that, there exist $ f(r)>0 $ and $ b(r), $ such that
\begin{align}\label{Phiranders}
	 \phi(r,z)=\Phi^2(r,z)&= \left(f(r)\sqrt{g(r)z^2+1} + b(r)z\right)^2,\end{align} 
substituting \eqref{Phiranders} in \eqref{eqW}, we have \begin{align*}
	W_z-zW_{zz}&=\frac{1}{2r}\frac{b'(r)}{f(r)(g(r)z^2+1))^{3/2}},
\end{align*} which implies $ b(r)=b$ is a  constant.

The last case, namely $ h_1\neq 0 $ and $ h_2(r)+h_3(r)= 0, $ is discarded because the condition \eqref{eqPhi} gives $ z^2\Phi_{zz}=\Phi-z\Phi_z, $ which implies that $ \Phi-z\Phi_z=0 $ for all $ y=(0,\overline{y}). $

Conversely, suppose that $ \phi $ is of the form $ \phi(r,z)=({h(r)})^{-2} {G}^2\left(h(r)z\right), $ where $ h $ and $ G $ are arbitrary positive differentiable functions. Substituting $ \phi $ into \eqref{defUVW}, we have 
\begin{align}
	U=&z\frac{(\ln(h))'}{2r},\label{Udouglas}\\
	V=W=&-\frac{(\ln(h))'}{2r},\label{Vdouglas}
\end{align}
 then
 \begin{align*}
 	R=&z\frac{n}{n+2}\frac{(\ln(h))'}{r},       \\
 	T=&-\frac{2}{n+2}\frac{(\ln(h))'}{r}.
 \end{align*}
Therefore $ D^A_{BCD}=0. $

On the other hand, it is known that a Randers metric $ F=\alpha+\beta $ is a Douglas metric if, and only if, $ \beta $ is a closed 1-form \cite{BaMa1997}, thus \eqref{Fgneq0} is a Douglas metric.

Finally, Proposition \ref{propdconvex} gives us the conditions \eqref{condGG}-\eqref{condGG2}.

\end{proof}

\begin{remark}
	Observe that, in  \eqref{eqPhi}- \eqref{Phiranders} we used strongly the positive definiteness of $ F. $
\end{remark}
\begin{remark}\label{remark001}
Observer that, when $ f(r)=\frac{k^2}{\sqrt{g(r)}}, $ \eqref{Fgneq0}  becomes \eqref{Fgeq0} with $ h(r)=g(r) $ and $ G(t)=k^2\sqrt{t^2+1} + bt. $
\end{remark}
The next corollary of Theorem \ref{theoDouglas} is useful  to construct new examples of Douglas metrics (see example \ref{example3}). 
\begin{corollary}\label{corollaryD}
Let $ {G_c}:\mathbb{R}\rightarrow\mathbb{R} $ be defined by
\begin{align}\label{eq:Ginth}
	{G_c(t)}=\int_0^{t}\left(\int_0^{\tau} h(\nu) {d}\nu\right) d{\tau} + c, 
\end{align}
where $ c $ is a real number and $ h:\mathbb{R}\rightarrow\mathbb{R} $ is a positive differentiable real function, such that $ \int_0^t \tau h(\tau) \operatorname{d\tau} $ is bounded above by $ L\in \mathbb{R}, $ with $ c-L>0. $ Then $ G_c $ is positive and the following Finsler metric,
	\begin{align*}
		F(x,y)=\vert\overline{y}\vert (g(r))^{-1}G_c(g(r)z),
	\end{align*}
	is of Douglas type. Here $ r=\vert\overline{x}\vert, $ $ z=\frac{y^0}{\vert\overline{y}\vert}$ and $  g:[0,\rho) \rightarrow \mathbb{R}$ is an arbitrary positive differentiable function of $ r $.
\end{corollary}
\begin{proof}	 
	For $ {G_c} $ defined by \eqref{eq:Ginth}  we have,
	\begin{align*}
		{G_c}''=&h(t) >0,\\
		{G_c}-t{G_c}'= & -\int_0^t \tau h(\tau)dt + c>-L+c> 0. 
	\end{align*}
Then, we only need to check the positivity of $ G_c. $ From the hypothesis, and integrating by parts, we have
	\begin{align} \label{desigGc}
		t\int_0^t h(\tau)\operatorname{d\tau}-\int_0^t\left(\int_0^{\tau}h(\nu)d\nu\right)\operatorname{d\tau}=\int_0^t \tau h(\tau) \operatorname{d\tau} < L,   \quad t\in \mathbb{R}.
	\end{align}
	The quantities $ t $ and $ \int_0^th(\tau)\operatorname{d\tau}$ have the same sign. Then \eqref{desigGc} implies
	\begin{align*}
		-L\leq t\int_{0}^{t}h(\tau)\operatorname{d\tau}-L<{G_c}-c.
	\end{align*}
	
\end{proof}

The Berwald tensor $ B=B^{A}_{BCD}\partial_A\otimes dx^B\otimes dx^C\otimes dx^D $ is defined by
\begin{align*}
	B^{A}_{BCD}:=&\frac{\partial^3 G^A}{\partial y^B\partial y^C\partial y^D}.
\end{align*}
A Finsler metric is called a Berwald metric if $ B^A_{BCD}=0, $ i.e. the spray coefficients $ G^A=G^A(x,y)$ are quadratic in $ y\in T_x{M} $ at every point $ x\in M. $

\begin{theorem}\label{theoBer}
Let $F=\vert\overline{y}\vert\sqrt{\phi(z,r)}$ be a  Finsler metric defined on $ M $, where $ z=\frac{y^0}{\vert\overline{y}\vert}, $ $r=\vert\overline{x}\vert$, and  $TM $  with coordinates   \eqref{coordx}, \eqref{coordy}. Then the
	Berwald curvature of $F$ is given by
\begin{align*}
		B^0_{000}=&\frac{s}{u} E_{zzz},\\
		B^0_{00l}=&\frac{1}{u}\left(E_{zz}\right)x^i-\frac{s}{u}\left(E_{zz}+zE_{zzz}\right)y^i,\\
		B^0_{0kl}=&sz\left({z}E_{zzz}+{3}{E_{zz}}\right)\frac{y^ky^l}{u^3} -zE_{zz}\left(\frac{x^ky^l+x^ly^k}{u^2}\right)-szE_{zz}\frac{\delta^{kl}}{u},\\
		B^0_{jkl}=&\frac{s}{u^4}\left[E-zE_z-2z^2E_{zz}-\frac{1}{3}z^3E_{zzz}\right]{(y^jy^ky^l)_{\overrightarrow{jkl}}}\\
		&+\frac{1}{u^3}\left[-E+zE_z+z^2E_{zz}\right](y^jy^kx^l)_{\overrightarrow{jkl}}+\frac{1}{u}\left[E-zE_z\right](\delta^{jk}x^l)_{\overrightarrow{jkl}}\\
		&+\frac{s}{u^2} \left[-E+zE_z+z^2E_{zz}\right](\delta^{jk}y^l)_{\overrightarrow{jkl}}\\
		B^i_{000}=& \frac{s}{u^2}H_{zzz}y^i-\frac{1}{u}W_{zzz}x^i\\
		B^i_{00l}=&-\frac{s}{u^3}(2H_{zz}+zH_{zzz})y^iy^l + \frac{1}{u^2}H_{zz}y^ix^l + \frac{z}{u^2}W_{zzz}y^lx^i + \frac{s}{u}H_{zz}\delta^{il},\\
		B^i_{0kl}=&\frac{s}{u^4}\left(3H_z+5zH_{zz} + z^2H_{zzz}\right)y^iy^ky^l-\frac{s}{u^2}(H_z+zH_{zz})(\delta^{kl}y^i)_{\overrightarrow{ikl}}\\
		&-\frac{1}{u^3}(H_z+zH_{zz})(x^ly^ky^i+x^ky^ly^i) + \frac{1}{u}H_z(x^l\delta^{ik} + x^k\delta^{jl})\\
		&+\frac{1}{u^3}(W_z-zW_{zz}-z^2W_{zzz})x^iy^ky^l-\frac{1}{u}(W_z-zW_{zz})x^i\delta^{kl},\\
		B^i_{jkl}=&\frac{s}{u^5}(-15zH_z-9z^2H_{zz}-z^3H_{zzz})y^iy^jy^ky^l \\
		&+(3zH_z+z^2H_{zz})\left(\frac{1}{u^4}y^ix^jy^ky^l +\frac{s}{u^3}\delta^{ij}y^ky^l+\frac{s}{u^3}\delta^{jl}y^ky^i \right)_{\overrightarrow{jkl}} \\
		& - \frac{1}{u^2}zH_z\left((x^jy^i+us\delta^{ij})\delta^{kl}+(x^j\delta^{ik}+x^k\delta^{ij})y^l\right)_{\overrightarrow{jkl}}\\
		&- \frac{1}{u^4}(3zW_z-3z^2W_{zz}-z^3W_{zzz})x^iy^jy^ky^l+ \frac{z}{u^2}(W_z-zW_{zz})x^i(y^j\delta^{kl})_{\overrightarrow{jkl}} ,
\end{align*}
	where $ s=\frac{\langle\overline{x},\overline{y}\rangle}{\vert\overline{y}\vert}, $ $ u=\vert\overline{y}\vert $,
\begin{align}\label{defEF}
	E&=U+zV,& H&=V+W,
\end{align}

\begin{align*}
	U=&\frac{2\phi\phi_{zr}-\phi_z\phi_r}{2r(2\phi\phi_{zz}-\phi_z^2)},&
	V=&\frac{\phi_r\phi_{zz}-\phi_z\phi_{zr}}{2r(2\phi\phi_{zz}-\phi_z^2)},&
	W=&\frac{\phi_r}{2r(2\phi-z\phi_z)},
\end{align*}
 and $(.)_{\overrightarrow{jkl}}$ denotes cyclic permutation.
\end{theorem}

\begin{proof}
	The arguments are similar to the proof of Theorem \ref{Dcurv} (compare \eqref{eq:G0}-\eqref{eq:Gi} with \eqref{eq:Giminus}-\eqref{eq:Giminusi}).
	
\end{proof}

The Landsberg tensor $L=L_{ABC}dx^A\otimes dx^B\otimes dx^C$ is defined by

\begin{align*}
	L_{ABC}:=-\frac{1}{4}[F^2]_{y^D}[G^D]_{y^Ay^By^C}.
\end{align*}
A Finsler metric is called a Landsberg metric if $ L_{ABC}=0. $ 

\begin{theorem}
	With the same notations of Theorem \ref{theoDouglas}. The
	Landsberg curvature of $F$ is given by
	\begin{align*}
	L_{000}=&-\frac{s}{4}(\phi_zE_{zzz}+\Omega(H_{zzz}-W_{zzz})),\\
	L_{00l}=&-\frac{1}{4}(\phi_zE_{zz}+\Omega  H_{zz})x^l + \frac{s}{4u}\left(\phi_z(E_{zz}+zE_{zzz})+\Omega(H_{zz}+zH_{zzz}- zW_{zzz})\right)y^l\\
	L_{0kl}=&	-\frac{s}{4u^2}\left(z\phi_z(zE_{zzz} + 3E_{zz}) + \Omega(H_z+3zH_{zz} + z^2H_{zzz}) +\right.\\
	&\left.\qquad\qquad +\Omega(W_z-zW_{zz}-z^2W_{zzz})\right)y^ky^l\\
	&+\frac{z}{4u}\left(\phi_zE_{zz}+\Omega H_{zz}\right)(x^ly^k+x^ky^l)\\
	&+\frac{s}{4}\left(z\phi_zE_{zz} + \Omega(H_z+zH_{zz}+W_z-zW_{zz})\right)\delta^{kl}\\
	L_{jkl}=&-\frac{s}{4u^3}\left({\phi_z}[3E-3zE_z-6z^2E_{zz} - z^3E_{zzz}]   \right.\\
	&\left. \qquad\quad -\Omega[6zH_z+6z^2H_{zz}+z^3H_{zzz}]- \Omega[3zW_z-3z^2W_{zz}-z^3W_{zzz}]\right){y^jy^ky^l}\\
	&-\frac{1}{4u^2}\left(\phi_z[-E+zE_z+z^2E_{zz}] + \Omega[zH_z+z^2H_{zz}]\right)\left(x^jy^ky^l\right)_{\overrightarrow{jkl}}\\
	&-\frac{s}{4u}\left(\phi_z[-E+zE_z+z^2E_{zz}] + \Omega[2zH_z+z^2H_{zz}+zW_z- z^2W_{zz}]\right)\left(y^l\delta^{jk}\right)_{\overrightarrow{jkl}}\\
	&-\frac{1}{4}\left(\phi_z[E-zE_z] - z\Omega H_z\right)\left(x^l\delta^{jk}\right)_{\overrightarrow{jkl}}.
	\end{align*}
\end{theorem}
\begin{proof}
	Apply directly the results of Theorem \ref{theoBer} into
	\[ [F^2]_{y^A}B^A_{BCD}=\vert\overline{y}\vert\phi_zB^0_{BCD} + \Omega y^iB^i_{BCD}. \]
	
\end{proof}

\begin{corollary}\label{corollaryunicorn}
	Let $ F=\vert\overline{y}\vert\sqrt{\phi(z,r)} $ be a Finsler metric. Then $ F $ is Berwald metric if, and only if, $ F $ is Landsberg metric. In this case, either $ F $ is Riemannian or $ F $ is of the form $ 
		F(x,y)=\vert\overline{y}\vert {(h(r))^{-1}} {G}\left(h(r)z\right),
	 $ where $ h $ and $ G $ are positive real functions such that $ G-zG'>0 $ and $ G''>0. $
\end{corollary}

\begin{proof}

	 Similar arguments used in the proof of Theorem \ref{Dcurv} show us that 

 $ L_{jkl}=0 $ if, and only if, 
		\begin{align}
		\phi_zE_{zz}+\Omega H_{zz}&=0,\label{Cond1L}\\
		 \phi_z[E-zE_z]-\Omega z H_z&=0, \label{Cond2L}\\
		 H_z+W_z-zW_{zz}&=0.\label{Cond3L}
	\end{align}
Suppose that $ F $ is a Landsberg metric. Taking the derivative of \eqref{Cond2L} with respect to $ z, $ we obtain\begin{align*}
2\phi\frac{H_z}{\phi_z}[\phi_z-z\phi_{zz}]=0.
\end{align*} 
If $ H_z=0, $ then, by \eqref{Cond1L} and \eqref{Cond3L}, $ E-zE_z=W_z-zW_{zz}=0, $ witch implies that $ B^A_{BCD}=0, $ i.e., $ F $ is a Berwald metric. If $ \phi_z-z\phi_{zz}=0 $, then $ F $ is a Riemannian metric.

	It is know that every Berwald  metric is a Douglas metric and a Randers metric $F=\alpha+ \beta $ is a Berwald metric, if and only if, $ \beta $ is parallel with respect to $ \alpha$ \cite{Kikuchi1979}.
Thus, we conclude that  the Douglas metric given by \eqref{Fgneq0} is a Berwald metric if, and only if, $ b=0 $ or $ \left(f^2(r){g(r)}\right)'=0. $  On the other hand, we can verify, using \eqref{Udouglas} and \eqref{Vdouglas}, that the Douglas metric given by \eqref{Fgeq0}, satisfies $ B^A_{BCD}=0. $ Thus, we conclude that, either $ F  $ is a Riemannian metric or, from Remark \ref{remark001}, $ F $ is of the form \eqref{Fgeq0}.
\end{proof}

The proof of Corollary \ref{corollaryunicorn} gives a classification of non Berwald metrics with vanishing Douglas curvature. 
\begin{corollary}
		Let $F=\vert\overline{y}\vert\sqrt{\phi(z,r)}$ be a  Finsler metric defined on $ M $, where $ z=\frac{y^0}{\vert\overline{y}\vert}, $ $r=\vert\overline{x}\vert$, and  $TM $  with coordinates   \eqref{coordx}, \eqref{coordy}. Then $ F $ is a non Berwald metric with vanishing Douglas curvature if, and only if, there exist positive differentiable functions $ f,g:[0,\rho) \rightarrow \mathbb{R},$ and a constant $ b,$ such that $ F $ is of the form
	\begin{align*}
		F(x,y)=f(r)\sqrt{g(r)(y^0)^2 + \vert\overline{y}\vert^2} + b y^0,
	\end{align*}
where  $ \left(f^2(r){g(r)}\right)'\neq 0, $ and $ 0<b^2<f ^2(r)g(r).$ 
\end{corollary}

\vspace{1cm}

It is known that a Finsler metric $ F $ is locally projectively flat if, and only if,
\begin{align}\label{rap}
	F_{x^Ay^B}y^A - F_{x^B}=0.
\end{align}
\begin{theorem}\label{teoflat}
	Let $F=\vert\overline{y}\vert\sqrt{\phi(z,r)}$ be a  Finsler metric defined on $ M $, where $ z=\frac{y^0}{\vert\overline{y}\vert}, $ $r=\vert\overline{x}\vert$, and  $TM $  with coordinates   \eqref{coordx}, \eqref{coordy}. Then $ F $ is locally projectively flat if, and only if,
	\[\phi_r=0.\]
	Moreover, any such metric is given by
	\[F(x,y)=\vert\overline{y}\vert G(z),\]
	where $ G:\mathbb{R}\rightarrow\mathbb{R} $ is an arbitrary positive differentiable real function such that \begin{align}
		G-zG'&>0,\label{1posi}\\
		G''&>0.\label{2posi}
	\end{align}
\end{theorem}
\begin{proof}	
	We have that,
	\begin{align*}
		F_{x^0}=&0,\\
		F_{x^iy^0}=&\frac{1}{r}[\sqrt{\phi}]_{rz}x^i,\\
		F_{x^iy^j}=&\frac{s}{r}([\sqrt{\phi}]_r-z[\sqrt{\phi}]_{rz})x^iy^j.
	\end{align*}
	Then, from \eqref{rap}, $ F $ is locally projectively flat if, and only if,
	\begin{align*}
		\frac{\vert\overline{y}\vert}{r}s[\sqrt{\phi}]_{rz}=&0,\\
		\frac{s}{r}([\sqrt{\phi}]_r-z[\sqrt{\phi}]_{rz})y^j - \frac{\vert\overline{y}\vert}{r}[\sqrt{\phi}]_{r}x^j=&0.
	\end{align*}
	for every $ y. $
	Finally, Proposition \ref{propdconvex} give the conditions \eqref{1posi}-\eqref{2posi}.
	
\end{proof}

In the next proposition, we recall the classification of the Ricci-flat metrics given by the authors in  \cite{marcal2020}. 
\begin{proposition}[\cite{marcal2020}]\label{propmarcal}
	For $ n\geq 3, $ $ F=\vert\overline{y}\vert\sqrt{\phi(z,r)} $ is Ricci-flat if, and only if, 
	\begin{align}
		(2r^2W+1)(U_z+nV+(n-3)W) + 2(nW+rW_r+r^2W_z(zW-U))& =0, \label{P}\\
		2U(U_{zz} + nV_z+(n-2)W_z)-\frac{1}{r}(U_{rz}+nU_r+(n-3)W_r)+&\label{Q}\\
		+nV(V+2W) + W((n-5)W+2zW_z) + U_z(2W-U_z)&=0.\nonumber
	\end{align}
\end{proposition} 
Theorem \ref{theoDouglas} and Proposition \ref{propmarcal} give a characterization for a Ricci-flat Finsler metric  with vanishing Douglas curvature.

\begin{corollary}\label{DouglasRicci}
	Let $F=\vert\overline{y}\vert\sqrt{\phi(z,r)}$ be a  Finsler metric defined on $ M=\mathbb{R}\times \mathbb{B}^n(\rho) $, where $ z=\frac{y^0}{\vert\overline{y}\vert}, $ $r=\vert\overline{x}\vert$, and  $TM $  with coordinates   \eqref{coordx}, \eqref{coordy}. Then $ F $  is Ricci-flat  with vanishing Douglas curvature if, and only if, either 
	\begin{enumerate}
		\item $ F $ is a Ricci-flat Riemannian metric,
		\item or, \[ F=\vert\overline{y}\vert r^CG\left(\frac{z}{r^C}\right),  \text{ $ C\in \mathbb{R}, $ when $ n=2,$  }\]   
		\item or, 
		  \[F=\vert\overline{y}\vert r^2G\left(\frac{z}{r^2}\right),  \text{ when $ n\neq 2,$ }\] 
		\item or  \[ F=\vert\overline{y}\vert G(z),\]
	\end{enumerate} 
 where $ G=G(t) $ is a positive real function such that $ G-tG'>0 $ and $ G''>0. $
\end{corollary}

\begin{proof}
	In \cite{Tian2012}, Y. Tian and X. Cheng, proved that a Randers metric $ F=\alpha +\beta $  is a Ricci-flat Douglas metric if, and only if, $ \beta $ is parallel with respect to $ \alpha $ and $ \alpha $ is Ricci-flat. In this case, $ F $ is a Berwald metric.
	 Then, the Douglas metric given by \eqref{Fgneq0}, is Ricci-flat if, and only if, either $ b=0 $ or $ \left(f^2(r)g(r)\right)'=0 $ and, in both cases, the Riemannian metric $ f(r)\sqrt{g(r)z^2+1} $ is Ricci-flat. From Remark \ref{remark001}, the case $ \left(f^2(r)g(r)\right)'=0 $ is included in the next part of the proof.
	 
Now, suppose that $ F $ is of the form \eqref{Fgeq0}, i.e., $ \phi=(h(r))^{-2}G^2(h(r)z). $ Using \eqref{Udouglas} and \eqref{Vdouglas} in  \eqref{P}-\eqref{Q}, we have that the positive function $ h $ satisfies the system
	\begin{align}
		(1-n)rh'^2+rhh''+(2n-3)hh'&=0\label{eqhricc}\\
		(n-2)\left(rh''-h'\right)&=0.\label{eqhricc2}
	\end{align}
If $ n=2, $ from \eqref{eqhricc}, there exist constants $ k$ and $C $ such that $ h(r)=kr^C. $
  If $ n\neq2$, from \eqref{eqhricc}-\eqref{eqhricc2}, we have that there exist constants $ k_1, k_2 $ such that $ h(r)= k_1r^2 + k_2$ with $ k_1k_2=0.$ Finally, Proposition \ref{propdconvex} gives the conditions $ G-tG'>0 $ and $ G''>0 $ of the positive function $ G. $

\end{proof}

  P. Marcal and Z. Shen  characterized Ricci-flat Riemannian warped metrics in \cite{marcal2020}.

\section{Douglas metric examples}

Note that the functions $ f,g$ and $ h $  in Theorem \ref{theoDouglas} are arbitrary, then we can consider the domain of $ F $ to be of the form $ \mathbb{R}\otimes \overline{M} $, where $ \overline{M} $ is an open symmetric subset of $ \mathbb{R}^n. $

\begin{example}
Considering $ {G}(t)=\sqrt{t^2 + \varepsilon} ,$ $ \varepsilon>0 $ in \eqref{Fgeq0}, the following Finsler metric 
\begin{align*}
	F(x,y)=\frac{\sqrt{g^2(\vert\overline{x}\vert)(y^0)^2 +\varepsilon \vert\overline{y}\vert^2}}{g(\vert\overline{x}\vert)},
\end{align*} 
is of Douglas type, where $ g $ is a positive arbitrary differentiable function.
	\end{example}

We observe that, if two functions satisfy \eqref{condGG} and \eqref{condGG2}, then the sum satisfies too. 

\begin{example}
	Considering ${G}={\sqrt{t^2 + \varepsilon}} +  \gamma t ,$ $ \varepsilon>0$ and $ \gamma^2<1 $ in \eqref{Fgeq0}, the following Randers metric 
	\begin{align*}
		F(x,y)= \frac{{\sqrt{g^2(\vert\overline{x}\vert)(y^0)^2+ \varepsilon\vert\overline{y}\vert^2} }}{g(\vert\overline{x}\vert)} + \gamma y^0,
	\end{align*} 
	is of Douglas type, where $ g $ is a positive arbitrary differentiable function.
\end{example}

\begin{example}
	Considering $ h(t)=\frac{3}{(t^2+1)^{5/2}} $ and $ c=1 $ in Corollary \ref{corollaryD}, the following Finsler metric
	\[F(x,y)=\vert\overline{y}\vert + \frac{2g^2(\vert\overline{x}\vert)(y^0)^2 + \vert\overline{y}\vert^2}{\sqrt{g^2(\vert\overline{x}\vert)(y^0)^2 + \vert\overline{y}\vert^2}}\]
	is of Douglas type, where $ g $ is a positive arbitrary differentiable function. Observe that, we can use any constant $ c\geq 1. $
\end{example}

\begin{example}
	Considering $ G(t)=\frac{2t^2+1}{\sqrt{t^2+1}}+2t $ in \eqref{Fgeq0}, the following Finsler metric
	\[F(x,y)=\displaystyle\frac{\left(\sqrt{g^2(\vert\overline{x}\vert)(y^0)^2+\vert\overline{y}\vert^2} + g(\vert\overline{x}\vert)y^0\right)^2}{\sqrt{g^2(\vert\overline{x}\vert)(y^0)^2 +\vert\overline{y}\vert^2}},\]
is of Douglas type, where $ g $ is a positive arbitrary differentiable function.
\end{example}
\begin{example}\label{example3}
	Considering $ h(t)={(t^2+1)^{-3}} $ and  $ c=1 $  in Corollary \ref{corollaryD}, the following Finsler metric
	\[ F(x,y)=  \frac{3y^0}{8}\arctan\left(g(\vert\overline{x}\vert)\frac{y^0}{\vert\overline{y}\vert}\right) + \frac{8g^2(\vert\overline{x}\vert)(y^0)^2 + 7\vert\overline{y}\vert^2}{8g(\vert\overline{x}\vert)\left(g^2(\vert\overline{x}\vert)(y^0)^2 + \vert\overline{y}\vert^2\right)}\vert\overline{y}\vert, \]
	is of Douglas type, where $ g $ is a positive arbitrary differentiable function. 
	Observe that, we can use any constant $ c\geq 1/4. $ 
\end{example}

\subsection*{Acknowledgements} I would like to thank Prof. Keti Tenenblat for helpful conversations during my postdoctoral project at the Universidade de Brasilia (Brazil). Her suggestions were invaluable. I am also grateful to the referees for their observations. 


\end{document}